\documentclass{compositio}

\usepackage{amsfonts,amssymb,latexsym,amsmath}
\usepackage{times}
\usepackage{dbnsymb}
\usepackage{mathrsfs}

\renewcommand{\AA}{\mathbb{A}}

\newcommand{\CC}{\mathbb{C}}
\newcommand{\DD}{\mathbb{D}}

\newcommand{\NN}{\mathbb{N}}

\newcommand{\PP}{\mathbb{P}}

\newcommand{\RR}{\mathbb{R}}

\newcommand{\VV}{\mathbb{V}}

\newcommand{\ZZ}{\mathbb{Z}}

\newcommand{\hH}{\mathcal{H}}
\newcommand{\nN}{\mathcal{N}}
\newcommand{\oO}{\mathcal{O}}

\newcommand{\rR}{\mathcal{R}}

\newcommand{\picbar}{\overline{\mathrm{Pic}}}

\newtheorem{lemma}{Lemma}
\newtheorem{proposition}[lemma]{Proposition}
\newtheorem{corollary}[lemma]{Corollary}
\newtheorem{theorem}[lemma]{Theorem}
\newtheorem*{thma}{Theorem A}
\newtheorem*{thmap}{Theorem A'}
\newtheorem*{thmb}{Theorem B}

\theoremstyle{definition}
\newtheorem{definition}[lemma]{Definition}

\theoremstyle{remark} 
\newtheorem*{remark}{Remark}
\newtheorem*{example}{Example}
\newtheorem{conjecture}[lemma]{Conjecture}

\DeclareMathOperator{\ord}{ord}

\title[Hilbert schemes of points on a locally planar curve and... its versal deformation]{Hilbert schemes of points on a locally planar curve \\
and the Severi strata of its versal deformation}

\author{Vivek Shende}
\address{Department of Mathematics, 
Massachusetts Institute of Technology,
Cambridge, MA 02139}

\begin{document}

\begin{abstract}
  Let $C$ be a locally planar curve.  Its versal deformation 
  admits a stratification
  by the genera of the fibres.  The strata are singular; we show
  that their multiplicities at the central point are determined by 
  the Euler numbers of the Hilbert schemes of points
  on $C$. 

  These Euler numbers have made two prior appearances.  
  First, in certain simple cases, they control the contribution of $C$ to
  the Pandharipande-Thomas curve counting invariants of three-folds.  
  In this context, our result identifies the strata multiplicities as the
  local contributions to the Gopakumar-Vafa BPS invariants.  
  Second, when $C$ is smooth away from
  a unique singular point, a conjecture
  of Oblomkov and the present author identifies the 
  Euler numbers of the Hilbert schemes
  with the ``U($\infty$)'' invariant of the link of the singularity.  
  We make contact with combinatorial ideas of Jaeger, 
  and suggest an approach to the conjecture. 
\end{abstract}

\maketitle

\section{Introduction}

Let $\mathcal{C} \to \Lambda$ be a projective flat 
family of integral, locally planar, 
complex algebraic curves over a smooth base.  The fibres
necessarily share the same arithmetic genus $g$, and it is known 
\cite{DH,T} that
the geometric genus gives a lower semicontinuous function 
$\tilde{g}:\Lambda \to \ZZ$.  For $ h \le g$ we write
\[\Lambda_h = \{ \lambda \in \Lambda\,|\, \mathcal{C}_\lambda \,
\mbox{is of geometric genus} \le h\}\]
This gives a stratification by closed subvarieties
\[\Lambda_0 \subset \cdots \subset
\Lambda_g = \Lambda\]
By semicontinuity, $\lambda \notin \Lambda_h$  unless 
$\tilde{g}(\lambda) \le h$. 
By convention we take $\Lambda_{h} = \emptyset$ for any $h > g$. 
 
We say the family is {\em locally versal at $\lambda \in \Lambda$} 
when the induced deformations of the germs of the singular points
of $\mathcal{C}_\lambda$ are versal.  We recall properties of
versal deformations of singularities in Section \ref{sec:smoothness} and
refer to \cite{GLS} for a detailed treatement. 
At a locally versal point $\lambda$, 
in the range $\tilde{g}(\lambda) \le h \le g$,
it is known that the stratum $\Lambda_h$ is nonempty of 
of pure codimension $g - h$ and is the closure of the 
locus $\Lambda^+_h$ of curves with $g-h$ nodes \cite{DH,T}. 
While $\Lambda^+_h$ is smooth, $\Lambda_h$ will
generally be singular.  We are interested in the multiplicities
$\deg_\lambda \Lambda_h$, i.e., 
the number of points near $\lambda$ in which $\Lambda_h$
intersects a generic space of the appropriate codimension.  For instance,
since $\Lambda_g = \Lambda$ is smooth by assumption, 
\begin{equation} 
  \label{eq:vg}
  \deg_{\lambda} \, \Lambda_{g} = 1
\end{equation}

The multiplicities depend
only on the singularities of $\mathcal{C}_\lambda$.  
Denote the germs of the singularities
by $c_i$, their respective contributions to $g - \tilde{g}(\lambda)$ 
by $\delta(c_i)$, 
and the bases of their miniversal deformations by $\VV(c_i)$. 
Fix $\pi: \Lambda \to \prod \VV(c_i)$ compatible with the 
deformations of $c_i$ induced by $\Lambda$; it is unique up to
first order and smooth by local versality of $\Lambda$ \cite[p. 237]{GLS}.
We write 
$\VV_h^+(c_i)$ for the locus in $\VV(c_i)$ where 
the fibres are smooth away from exactly $\delta(c_i) - h$ nodes, and
$\VV_h(c_i)$ for its closure. 
The stratifications are compatible: 
$\Lambda_{\tilde{g}+h}= 
\pi^{-1}\left( \bigcup_{h = \sum h_i} \prod_i \VV_{h_i}(c_i) \right) $.  
Since $\pi$ is a smooth morphism, 
\begin{equation} 
\label{eq:multprodformula}
\deg_{\lambda} \Lambda_{\tilde{g}+h} \,\, =\! 
 \sum_{h = \sum h_i} \prod_i \deg_{[c_i]}
 \VV_{h_i}(c_i) 
\end{equation}

In particular, the multiplicity of 
$\Lambda_{g-1}$ is given by the sum of the multiplicities of the 
discriminant loci in the $\VV(c_i)$.  These are the Milnor numbers
$\mu(c_i)$.  If
$c_i$ be the  germ of $f(x,y) = 0$ at $(0,0)$, then
for sufficiently general $g(x,y)$, the function
$f(x,y) + \epsilon g(x,y)$ 
has only simple critical points in a neighborhood of $(0,0)$
and moreover only one in each fibre $(f+\epsilon g)^{-1}(t)$.
The total number of critical points is by definition $\mu(c_i)$.
The family of
curves $C_{s,t} = \{(x,y)\,|\,f + s g = t\}$ induces a map from the germ at
zero in the $(s,t)$-plane to $\VV(c_i)$, 
and the line $s = \epsilon$ intersects
the discriminant locus in the $\mu(c_i)$ values of $t$ for which
$(f+sg)^{-1}(t)$ acquires a node.

Let $b(c_i)$ be the number of analytic local
branches. Milnor has shown \cite[Theorem 10.5]{M} that
$\mu(c_i) = 2 \delta(c_i) + 1 - b(c_i)$. 
Therefore
$\chi(\mathcal{C}_\lambda) = 2 - 2\tilde{g} + \sum (1- b(c_i)) = 
2 - 2g + \sum \mu(c_i)$, and so: 
\begin{equation}
  \label{eq:vgminus1}
 \deg_\lambda \Lambda_{g-1} = \chi(\mathcal{C}_\lambda) + 2g - 2
\end{equation} 

One expects that going to deeper strata will lead to 
increasingly difficult calculations.  
Nonetheless, Fantechi, G\"ottsche, and van Straten \cite{FGvS} 
showed that
the multiplicity of the deepest stratum  is equal to the
topological Euler number of the compactified Jacobian
$\picbar^0(\mathcal{C}_\lambda)$.  This space is described in detail 
in \cite{AK}; it
parameterizes torsion free, rank one, degree zero sheaves 
on $\mathcal{C}_\lambda$.
\begin{equation}
  \label{eq:v0}
  \deg_\lambda \Lambda_0 = \chi\left(\picbar^0(\mathcal{C}_\lambda)\right)
\end{equation}

Unless $\mathcal{C}_\lambda$ is rational, 
both sides of Equation (\ref{eq:v0}) vanish
-- the left because $\lambda \notin V_0$ by semicontinuity, 
and the right because 
the compactified Jacobian is topologically a product of the Jacobian of the
normalization of $\mathcal{C}_\lambda$ 
and factors coming from the singularities.
However, if $\bar{c}_i$ is a rational curve smooth
away from a singularity analytically isomorphic to $c_i$, then
Equations (\ref{eq:multprodformula}) and (\ref{eq:v0}) imply:
\begin{equation}
  \label{eq:vgg}
  \deg_\lambda \Lambda_{\tilde{g}(\lambda)} = \prod_i 
  \chi \left(\picbar^0(\overline{c}_i)\right)
\end{equation}

Our main result interpolates between
Equations (\ref{eq:vg}), (\ref{eq:vgminus1}), (\ref{eq:v0}), and (\ref{eq:vgg}).
We will need the Hilbert schemes of points, 
$X^{[n]} = \{ \mbox{ zero dimensional subschemes of $X$ of  length $n$ } \}$.

\begin{thma}
Let $\mathcal{C} \to \Lambda$ be a family
of complete, integral, locally planar curves of arithmetic genus $g$.  If
the family is locally versal at 
$\lambda \in \Lambda$, then:
\[
\sum_{n=0}^\infty q^{n} \chi( \mathcal{C}_{\lambda}^{[n]}) \,\, = \,
\sum_{h = \tilde{g}}^{g} q^{g-h} 
(1 - q)^{2h-2} \,\, \deg_\lambda \Lambda_h  \]
\end{thma}

There is an equivalent local version.  Let $c$ be the germ
of a plane curve singularity.  Fix a plane curve $C$ such that 
$c$ is the germ of $C$ at some point $p$.  We write
$c^{[n]}$ for the subvariety 
of $C^{[n]}$ whose closed points parameterize 
subschemes of $C$ which are 
set-theoretically supported at $p$.  This space 
depends only on the completion of $C$ at $p$.

\begin{thmap}
Let $c$ be the germ of
a plane curve singularity  which contributes $\delta$ to the arithmetic
genus and has $b$ analytic local branches. 
If $[c] \in \VV$ is the central point in the base of a versal deformation, then:
\[
\sum_{n=0}^\infty q^n \chi(c^{[n]}) = 
\sum_{h=0}^\delta q^{\delta-h} (1 - q)^{2h-b}
\deg_{[c]} \VV_h
\]
\end{thmap}

Theorem A restricts to Equations (\ref{eq:vg}) and (\ref{eq:vgminus1}); 
it implies Equations (\ref{eq:v0}) and (\ref{eq:vgg}) 
because $C^{[n]}$ is a $\PP^{n-g}$ bundle
over $\picbar^0(C)$ for large $n$ \cite{AK}.   The proof combines the methods of 
Fantechi, G\"ottsche, 
and van Straten \cite{FGvS}, 
techniques of Pandharipande and Thomas \cite{PT3},
and the following smoothness result.  
For a morphism $X \to Y$, we denote the relative Hilbert scheme by
$X^{[k]}_Y = \{ \,(y \in Y, [Z] \in X_y^{[k]})\,\}$.

\begin{thmb}
  Let $\mathcal{C} \to \Lambda$ be a family of complete, reduced, locally planar
  curves.  If the family is locally versal at $\lambda \in \Lambda$ and 
  $\lambda \in \DD^k \subset \Lambda$ is a generic, sufficently small
  $k$-dimensional polydisc, 
  then the total space of the relative Hilbert scheme 
  $\mathcal{C}_{\DD^k}^{[h]}$ is smooth if $h \le k$. 
\end{thmb}

We recall facts about generating series of Euler numbers of Hilbert schemes in 
Section \ref{sec:background}, prove Theorem A in 
Section \ref{sec:multiplicity} --
assuming Theorem B, which we prove in Section \ref{sec:smoothness}.  
We present formulas for the multiplicities in the case of ADE singularities 
in Section
\ref{sec:examples}.  
The final two sections discuss previous appearances 
of the series in the left hand side of Theorem A.  
In Section \ref{sec:bps}, we explain its relation to 
the contribution of $\mathcal{C}_\lambda$ 
to Gopakumar-Vafa invariants in Pandharipande-Thomas theory \cite{PT3}.  
Section \ref{sec:knots}
suggests that Theorem A may relate
a conjecture of Oblomkov and 
Shende \cite{OS} -- comparing
Euler numbers of Hilbert schemes of points on singular curves
to the HOMFLY polynomials of the 
links of the singularities -- to 
work of Jaeger on state-sum formulae for the HOMFLY polynomial \cite{Ja}.  
The reader is warned that the final section, in the words
of the anonymous reviewer, 
``is rather speculative, unfinished, and at best has the status
of a possible approach that might be tried.'' 

\begin{acknowledgements}
This work began as
attempt to answer a question of Rahul Pandharipande: ``how can you {\em see}
the curves'' counted by the local Gopakumar-Vafa invariants? 
I thank in addition Lotte Hollands, 
Sheldon Katz, Alexei Oblomkov, Giulia Sacc\`a, and 
Richard Thomas for helpful discussions, the anonymous reviewer for many 
corrections and 
suggestions, and the Instituto Superior T\'ecnico
of Lisbon for its hospitality during the completion of this work.
\end{acknowledgements}

\section{Background} \label{sec:background}

We need
the following properties of Euler numbers of complex varieties.
\begin{itemize}
\item $\chi(X \setminus Y) = \chi(X) -\chi(Y)$ for 
 $X \subset Y$ a closed immersion. 
\item $\chi(A \times B) = \chi(A) \chi(B)$.
\item $\chi(\AA^1) = 1$. 
\end{itemize}
The first property makes it natural to weight Euler numbers by 
constructible functions.  For $f$ a constructible 
function on a space $Z$, we write
$\chi(Z,f) := \sum_i  \chi(f^{-1}(i))\cdot i$. 

The remainder of the section collects for convenience facts 
about generating functions of Euler numbers of Hilbert schemes 
of locally planar 
curves.  All results are extracted from 
the work of Pandharipande and Thomas \cite[Appendix B]{PT3}.  
Unless otherwise specified, $C$ is an integral, Gorenstein curve 
of arithmetic genus $g$ and geometric genus $\tilde{g}$. 

The following statement is elementary. 

\begin{proposition} \label{prop:bpsform}
  Let $\mathbf{f} = (f_0, f_1, \ldots)$ be an arbitrary 
  sequence of integers.  
  Then there is a unique sequence
  of integers $\mathbf{n} = 
  (n_g, n_{g-1}, \ldots)$ giving an equality of formal power series 
  \[\sum_{d=0}^\infty f_d q^{d} = \sum_{h= -\infty}^g n_h q^{g-h} (1-q)^{2h-2} \]
  The matrix $T = T(g)$ such that $\mathbf{n} =
  T \mathbf{f}$ is lower triangular with ones on the diagonal; in particular
  \begin{eqnarray*}
    n_g & = & f_0 \\
    n_{g-1} & = & f_1 + (2g-2)f_0 
  \end{eqnarray*}
  The $n_h$ vanish for $h < 0$ if and only if $f_d - f_{2g-2-d} = c \cdot (d+1-g)$ 
  for
  some 
  $c$, in which case $c = n_0$. 
\end{proposition}

\begin{definition} \label{def:bps}
  Let $n_h(C) \in \ZZ$ be defined by   
  \[
  \sum_{n=0}^\infty q^{n} \chi( C^{[n]}) \,\, = \,\,
  \sum_{h = -\infty}^{g} q^{g-h} (1-q)^{2h-2}  \,\, n_h(C) \]  
\end{definition}

\begin{lemma} \label{lem:macd}
  For $C$ a smooth curve,
  $\sum q^n \chi(C^{[n]}) = (1-q)^{-\chi(C)}$.  
  If in addition $C$ is proper of genus $g$, then
  $n_g(C) = 1$ and
  $n_{h}(C) = 0$ for $h \ne g$.
\end{lemma}
\begin{proof}
  For a smooth curve, the Hilbert schemes and symmetric products agree.
  Thus the claim follows from Macdonald's calculation of
  the cohomology of symmetric products of curves \cite{Ma}. 
\end{proof}

\begin{remark}  The assertion of 
Theorem A can be restated as
$n_h(\mathcal{C}_\lambda) = \deg_{\lambda} \Lambda_h$. 
\end{remark}

\begin{lemma} \label{lem:hart}  (Hartshorne \cite{H}.)
  Let $F$ 
  be a torsion free sheaf on $C$. 
  Write $F^*$ for $\hH om (F,\oO_C)$. 
  Then $\mathcal{E}xt^{\ge 1}(F,\oO_C) = 0$ and 
  $F = (F^*)^*$.  Serre duality
  holds in the form $\mathrm{H}^i(F) = \mathrm{H}^{1-i}(F^* \otimes \omega_C)^*$.
  For $F$ rank one and torsion free, define its 
  degree $d(F):=\chi(F) - \chi(\oO_C)$.   This satisfies 
  $d(F) = - d(F^*)$, and, for $L$ any line bundle, 
  $d(F \otimes L) = d(F) + d(L)$. 
\end{lemma}

\begin{proposition}  \label{prop:vanishing}
  We have 
  $n_h(C) = 0$ for $h < 0$.
  Moreover, $n_0(C)$ is the Euler number of the compactified
  Jacobian of $C$. At the other extreme
  $n_{g-1}(C) = \chi(C) + 2g - 2$ and
  $n_g(C) = 1$. 
\end{proposition}
\begin{proof}
  Let $\picbar^n(C)$ be the 
  moduli of rank one torsion free sheaves of degree $n$.  
  There  is a map $AJ_n: C^{[n]} \to \picbar^n(C)$ taking a subscheme $Z \subset C$
  to $I_Z^*$, the dual
  of the ideal sheaf cutting it out \cite{AK}.  The inclusion 
  $I_Z \to \oO_C$ dualizes to a section $\oO_C \to I_Z^*$, thus the 
  fibre $AJ_n^{-1}(F) = \PP(\mathrm{H}^0(F))$.  Viewing
  $\mathrm{h}^0: [F] \mapsto \mathrm{h}^0(F)$ 
  as a constructible function on $\picbar^n(C)$, we have
  $\chi(C^{[n]}) = \chi(\picbar^n(C), \mathrm{h}^0)$.
  The involution $F \mapsto \omega_C \otimes F^*$ induces
  an isomorphism $\iota: \picbar^n(C) \cong \picbar^{2g - 2 -n}$. By Serre duality,
  $\iota \circ \mathrm{h}^0 = \mathrm{h}^1$, and by Riemann-Roch,
  \[\chi(\picbar^n(C), \mathrm{h}^0)
  - \chi(\picbar^{2g-2-n}(C), \mathrm{h}^0) = 
  \chi(\picbar^n(C), \mathrm{h}^0 - \mathrm{h}^0 \circ \iota) =
  (n + 1 - g) \chi(\picbar^n(C))
  \]
  The choice of a degree $1$ line bundle induces isomorphisms 
  $\picbar^n(C) \cong \picbar^{n+1}(C)$, hence these
  spaces have the same  Euler number.  The result now follows
  from Proposition \ref{prop:bpsform}.
\end{proof}

\begin{corollary}\label{cor:nodalcubic}
  Let $\PP^1$, $\PP^1_{\mathrm{node}}$, $\PP^1_{\mathrm{cusp}}$ be rational curves
  that are smooth, have one node, and have one cusp respectively.  
  \begin{itemize}
    \item $n_0(\PP^1) = 1$ and all other $n_h$ vanish. 
    \item $n_0(\PP^1_{\mathrm{node}}) = 1$ and $n_1(\PP^1_{\mathrm{node}}) = 1$ 
      and all other $n_h$ vanish.
    \item $n_0(\PP^1_{\mathrm{cusp}}) = 2$ and $n_1(\PP^1_{\mathrm{cusp}}) = 1$ 
      and all other $n_h$ vanish. 
  \end{itemize}
\end{corollary}
\begin{proof}
  Follows from the ``in particular'' of Proposition \ref{prop:bpsform} and
  the vanishing of Proposition \ref{prop:vanishing}. 
\end{proof}

More can be said by working locally at the singularities.  

\begin{definition} \label{def:localbps}
  Let $c$ be the germ of a Gorenstein curve singularity, 
  let $\delta$ be its delta invariant,
  and $b$ the number of analytic local branches.  Define
  $n_h(c)$ by the formula 
  \[
  \sum_{h=-\infty}^\delta q^{\delta -h} (1 - q)^{2h}\,
  n_h(c)
  =  (1-q)^{b} \sum_{n=0}^\infty q^n \, \chi(c^{[n]})
  \]
\end{definition}

\begin{remark}  Theorem A' asserts that when $c$ is
planar,
$n_h(c) = \deg_{[c]} \VV_h$. 
\end{remark}

\begin{proposition} \label{prop:bpsmultiplicativity}
  If $C$ has singularities
  $c_1, \ldots c_k$ and geometric genus
  $\tilde{g}$,
  \[n_h(C) \,\,\,\, = \!\!\!\!\!\!\! 
  \sum_{i_1 + \ldots + i_k + \tilde{g} = h} \!\!\!\!\!\!
  n_{i_1}(c_1)\cdots n_{i_k}(c_k)\]
\end{proposition}
\begin{proof}
  Stratifying the Hilbert scheme of $C$ by the number of points
  at each of the $c_i$, we see
  \begin{eqnarray*}
    \sum_{n=0}^\infty q^n \chi(C^{[n]}) & = &
    \left(\sum_{n=0}^\infty q^n \chi((C \setminus 
      \coprod c_i)^{[n]})\right) \prod_i \left(
      \sum_{n=0}^\infty q^n \chi(c_i^{[n]})\right)  \\ 
    & = &
    (1-q)^{2\tilde{g}-2 + \sum b(c_i)} \prod_i 
    \sum_{n=0}^\infty q^n \chi(c_i^{[n]})
  \end{eqnarray*}
  Substituting in the definitions of the $n_h$, 
  \[ \sum_{h = 0}^g q^{g-h} (1-q)^{2h-2} n_h(C) = 
  (1-q)^{2\tilde{g}-2}  \prod_i \sum_{h=-\infty}^\delta 
  q^{\delta -h} (1 - q)^{2h} n_h(c)
  \]
  Collecting terms and writing $z^2 = q^{-1} (1-q)^2$, 
  \[ \sum_{h = 0}^g z^{2h} n_h(C) = 
  z^{2\tilde{g}}  \prod_i \sum_{h=-\infty}^\delta 
   z^{2h}  n_h(c)
  \]
  Comparing coefficients of $z$ yields the result.
\end{proof}

\begin{corollary} \label{cor:inrational} 
  If $C$ is a rational curve with a single singularity
  $c$, then $n_h(C) = n_h(c)$.  
\end{corollary} 
\begin{corollary} \label{cor:singsupport}
  For $c$ the germ of a plane curve singularity, 
  $n_h(c)$
  vanishes for $h < 0$. 
\end{corollary}
\begin{corollary} \label{cor:support}
  For $C$ a complete, locally planar curve, $n_h(C)$ vanishes 
  for $h < \tilde{g}(C)$.
\end{corollary}

\begin{corollary} \label{cor:nodalcalc}
  If $C$ is nodal of geometric genus $\tilde{g}$ and
  arithmetic genus $g$, then 
  $n_h(C) = {g-\tilde{g} \choose g-h}$. 
\end{corollary}

\begin{corollary} \label{cor:equiv} The following are equivalent:

\vspace{1mm}
\noindent Theorem A:
    For a family $\mathcal{C} \to \Lambda$ 
    of integral
    locally planar curves locally versal at $\lambda$, 
    $n_h(C_\lambda) = \deg_\lambda \Lambda_h$.
\noindent Theorem A':  For  $c$ a plane curve singularity, 
    $n_h(c)\, = \deg_{[c]} \VV_h\,$.
\end{corollary}
\begin{proof}
The ``$A' \implies A$'' direction follows from comparing the relation between 
the multiplicities asserted in Equation \ref{eq:multprodformula}
with the relation between the $n_h$ established in 
Proposition \ref{prop:bpsmultiplicativity}.  To see
``$A \implies A'$'', consider locally 
versal deformations of curves with unique singular points.
\end{proof}

  We now remark on the relation between smoothness of relative Hilbert schemes
  and relative compactified Jacobians.  The result and its proof 
  are closely analogous to \cite[Theorem 4]{PT3}.

  \begin{proposition}
    Let $\mathcal{C} \to S$ be a family over a smooth base
    of complete integral Gorenstein
    curves of arithmetic genus $g$.  Then the following are equivalent: 
    \begin{enumerate}
    \item The total space of the relative Hilbert scheme 
      $\mathcal{C}^{[n]}_S$ is smooth for some $n \ge 2g-1$.      
    \item The total space of the relative compactified Jacobian
      $\overline{\mathrm{Pic}}^0(\mathcal{C}/S)$ is smooth.
    \item The total space of the relative Hilbert scheme  
      $\mathcal{C}^{[n]}_S$ is smooth for all $n \ge 2g-1$.
    \item The total space of the relative Hilbert scheme 
      $\mathcal{C}^{[n]}_S$ is smooth for all $n$.
    \end{enumerate}
  \end{proposition}
  \begin{proof} \label{prop:smoothcompare}
    It suffices to take $S$ a small polydisc.  
    As in Proposition \ref{prop:vanishing}, Riemann-Roch for
    Gorenstein curves ensures that the Abel-Jacobi map 
    $\mathcal{C}^{[n]} \to \overline{\mathrm{Pic}}^n(\mathcal{C}/S)$
    is a bundle with fibres $\PP^{n-g}$ once $n \ge 2g-1$.
    Choose a section of $S \to \mathcal{C}$ with image in the smooth locus
    of each fibre gives a line bundle of relative degree $1$ over $S$,
    to induce identifications
    between $\overline{\mathrm{Pic}}^n(\mathcal{C}/S)$ for varying $n$. 
    Thus $(i) \implies (ii)$ and $(ii) \implies (iii)$.   
    The section also induces an embedding 
    $\mathcal{C}^{[n]}_S \subset \mathcal{C}^{[n+1]}_S$.  For 
    $p \in \mathcal{C}^{[n]}_S$ corresponding to a subscheme 
    supported away from the section, some
    analytic neighborhood $p \in U \subset \mathcal{C}^{[n+1]}_S$ is
    analytically a product of  $\overline{U} = U \cap \mathcal{C}^{[n]}_S$
    with a disc.  Thus if $\mathcal{C}^{[n+1]}_S$ is smooth, so is
    $\overline{U}$.  By choosing different sections, we may cover 
    $\mathcal{C}^{[n]}_S$ with such neighborhoods. 
    Thus $(iii) \implies (iv)$.  It is clear that $(iv) \implies (i)$.
  \end{proof}

  \begin{corollary} \label{cor:smoothcompare}
    Let $\mathcal{C} \to \Lambda$ be a family of integral locally planar
    curves, locally versal at $\lambda \in \Lambda$.  Let $\delta$ be 
    the difference between the arithmetic and geometric genera of
    the curve $\mathcal{C}_\lambda$. Then for any $h$, any $k \ge \delta$,
    and
    any generic, sufficiently small $\lambda \in \DD^k \subset \Lambda$, 
    the relative Hilbert scheme $\mathcal{C}^{[h]}_{\DD^k}$ is smooth.
  \end{corollary}
  \begin{proof}
    It is shown in \cite{FGvS} that the relative compactified Jacobian
    over $\DD^k$ is smooth in this situation; more precisely, 
    smoothness holds once $T_\lambda \DD^k$ is transverse to the
    reduced tangent cone of the equigeneric stratum.  
    The result now follows from Proposition \ref{prop:smoothcompare}
  \end{proof}

\section{The proof of Theorem A} \label{sec:multiplicity}

  \begin{theorem}
    Let $\mathcal{C} \to \Lambda$ be a family of integral locally 
    planar curves.  Assume $\Lambda$ is locally versal at 
    $\lambda$.  Let $\Lambda_h$ be the locus of curves
    of geometric genus $\le h$.  Then 
    $n_h(\mathcal{C}_\lambda) = \deg_\lambda \Lambda_h$.
  \end{theorem}
  \begin{proof}  
    The RHS of this equality vanishes unless $\tilde{g} \le h \le g$ since in
    this case $\lambda \notin \Lambda_h$; the LHS vanishes as well by 
    Proposition \ref{prop:vanishing} and Corollary \ref{cor:support}.  
    So assume $\tilde{g} \le h \le g$.  Then 
    $\Lambda_h$  is of pure codimension $g-h$, 
    and is the closure of the locus $\Lambda^+_h$ of curves with
    $g-h$ nodes \cite{DH, T}.  

    Choose a small polydisc $\lambda \in 
    D = \DD^{g-h} \times \DD \subset \Lambda$ subject to:
    
    \begin{enumerate}
    \item  $D_0 := \DD^{g-h} \times \{0\}$ intersects $\Lambda_h$ 
      only at $\lambda$. 
    \item $D_\epsilon:= \DD^{g-h} \times \{\epsilon \}$ intersects 
      $\Lambda_h$ generically,
      i.e., at $\deg_\lambda \Lambda_h$ points of $\Lambda_h^+$.  That is,
      the points of intersection correspond to 
      nodal curves of genus $h$.
    \item The relative Hilbert schemes 
      $\mathcal{C}_D^{[i]}$, $\mathcal{C}_{D_0}^{[i]}$, $\mathcal{C}_{D_\epsilon}^{[i]}$
      are smooth for $i \le g-h$. 
    \end{enumerate}  
    
    Each of these conditions is generically true -- the third by Theorem B, 
    which we prove as Corollary \ref{cor:smooth} 
    in Section \ref{sec:smoothness} -- so we may satisfy them all
    simultaneously.   By condition (iii) above, $\mathcal{C}_{D_0}^{[i]}$
    and $\mathcal{C}_{D_\epsilon}^{[i]}$ are deformation equivalent smooth
    varieties 
    for $i \le g-h$.  In particular, they are diffeomorphic, and
    hence have the same Euler numbers.

    We define
    constructible functions 
    $n_i$, $\chi_i:\Lambda \to \ZZ$ by their values on the fibres: 
    for $p \in \Lambda$, 
    \begin{eqnarray*}
      n_i: p & \mapsto & n_i(\mathcal{C}_p) \\
      \chi_i: p & \mapsto & \chi(\mathcal{C}_p^{[i]})
    \end{eqnarray*}
    Observe that
    $ \chi(D_0, \chi_i) = \chi(\mathcal{C}_{D_0}^{[i]}) = 
    \chi(\mathcal{C}_{D_\epsilon}^{[i]})
    = \chi(D_\epsilon, \chi_i)$ for 
    $i \le g-h$.  But by Proposition \ref{prop:bpsform}, there is a linear
    change of variables between the 
    $\chi_0, \ldots, \chi_{g-h}$ and the $n_g, \ldots, n_h$.  Therefore,
    \[ \chi(D_0, n_j) = \chi(D_\epsilon, n_j) \,\,\,\,\,\,\,\mbox{for }
    g \ge j \ge h  \ \]
    As we know from Corollary 
    \ref{cor:support} that $n_h$ is supported on $\Lambda_h$, 
    \[ n_h(\mathcal{C}_\lambda) = \chi(D_0, n_h)  = \chi(D_\epsilon, n_j) = 
    \!\!\! \sum_{p \in D_\epsilon \cap \Lambda_{h}} \!\! n_h(\mathcal{C}_p) = \#D_\epsilon 
    \cap  \Lambda_h = \deg_\lambda \Lambda_h \]
    We have already explained the first two equalities.  The third holds
    again because $n_h$ is supported on $\Lambda_h$.  The fourth because 
    each $\mathcal{C}_p$ is nodal of geometric genus $h$ so 
    $n_h(\mathcal{C}_p) = 1$ by Corollary \ref{cor:nodalcalc}.  The final
    equality holds by definition of the multiplicity.
  \end{proof}

\section{Smoothness of relative Hilbert schemes} \label{sec:smoothness}
  Let $V \subset \CC[x,y]$ be a  finite dimensional smooth 
  family of polynomials, and consider the family of curves
  \[\mathcal{C}_V := \{(f \in V, p \in \CC^2)\,|\,f(p)=0\} 
  \subset V \times \CC^2\]
  We have
  $\mathcal{C}_V^{[k]} \subset V \times (\CC^2)^{[k]}$. The Hilbert
  scheme of points on a surface is smooth \cite{F}, and
  for
  $I \subset \CC[x,y]$ the tangent space is
  $T_I (\CC^2)^{[k]} =  \mathrm{Hom}_{\CC[x,y]}(I, \CC[x,y]/I)$, where
  a map $\eta$ corresponds to the tangent vector 
  $I(\eta) := \{\phi + \epsilon \phi' \,|\, 
  \phi \in I,\, \eta(\phi) = \phi' \bmod I \} \subset 
  \CC[x,y,\epsilon]/\epsilon^2$. 
  Writing $\tilde{I}$ for the image of $I$ in $\CC[x,y]/f$,  we have an
  exact sequence:
  \begin{equation} \label{eq:defseq}
  0 \to T_{(f,\tilde{I})}\mathcal{C}_{V}^{[k]} \to T_f V \times 
  T_I (\CC^2)^{[k]} 
  \xrightarrow{(f+\epsilon g,\,  \eta)\, \mapsto \, \eta(f) - g \bmod I}  
  \CC[x,y]/I
  \end{equation}
  If $f$ is squarefree, then all fibres in a neighborhood of
  $f \in U \subset V$ will be reduced, and the relative Hilbert
  schemes $\mathcal{C}_U^{[k]}$ are
  reduced, of  pure dimension $k + \dim V$, and locally 
  complete intersections
  \cite{BGS}.  Thus for squarefree $f$, the space 
  $\mathcal{C}_{V}^{[k]}$ is smooth
  at $(f,\tilde{I})$ if and only if 
  $\dim T_{(f,\tilde{I})} \mathcal{C}_{V}^{[k]} = k + \dim V$.  By counting dimensions
  this occurs  if and only if the final map of Sequence (\ref{eq:defseq}) 
  is surjective.  The easiest
  way to ensure this is to ask for surjectivity already at
  $\eta = 0$, i.e., that $T_f V \twoheadrightarrow \CC[x,y]/I$.

  We now recall basic notions from the deformation theory of singularities; 
  for details we refer
  to \cite{GLS}.  Let $(X,x)$ be the germ of a complex analytic space.  A
  deformation of $(X,x)$ is a flat morphism of
  germs of complex analytic spaces, $(\mathscr{X},x) \to (B,b)$, together 
  with an isomorphism from $(X,x)$ to the fibre over $b$.  A deformation
  $(\mathscr{X},x) \to (\VV,v)$ is said to be {\em versal} if given a
  flat morphism $(\mathscr{Y},y) \to (A,a)$, a closed subgerm 
  $(A',a) \subset (A,a)$, a map $\phi': (A',a)\to (\VV,v)$ and an 
  isomorphism of deformations 
  $(\mathscr{Y}|_{A'},y) \cong_{A'} (\mathscr{X}|_{A'},x)$, 
  there is a (nonunique) extension $\phi:(A,a)\to (\VV,v)$ of $\phi'$ which
  admits a compatible isomorphism $(\mathscr{Y},y) \cong_A (\mathscr{X}|_{A},x)$. 
  If the Zariski tangent map to $\phi$ {\em is} always uniquely determined
  by the given data, then $(\mathscr{X},x) \to (\VV,v)$ is said to be 
  {\em miniversal}.  The existence of a versal 
  deformation $(\mathscr{X},x)\to (\VV,v)$ guarantees the existence
  of a miniversal $(\overline{\mathscr{X}},\overline{x})\to 
  (\overline{\VV},\overline{v})$
  and moreover there are compatible isomorphisms
  $(\VV,v) \cong (\overline{\VV},\overline{v}) \times (\CC^k,0)$ and
  $(\mathscr{X},x) \cong (\overline{\mathscr{X}},\overline{x}) \times
  (\CC^k,0)$.

  The miniversal deformation of an isolated plane curve singularity
  has an explicit description.  Let $(C,0)$ be the germ at the 
  origin of the zero locus of 
  some $f \in (x,y)\CC[x,y]$.  Fix $g_1\ldots g_{\tau} \in \CC[x,y]$ 
  whose images form a basis of the vector
  space $\mathcal{T}^1 = \CC[x,y]/(f,\partial_x f, \partial_y f)$.  
  Then consider $F:\CC^{\tau} \times \CC^2 \to \CC^{\tau} \times \CC$
  given by $F(t_1,\ldots,t_{\tau},x,y) = 
  (t_1,\ldots,t_\tau,(f + \sum g_i t_i)(x,y))$.  Taking the fibre
  over $\CC^{\tau} \times 0$ gives a family of curves over
  $\CC^{\tau}$; taking germs at the origin gives the miniversal deformation
  $(\mathcal{C},0) \to (\CC^\tau,0)$ 
  of $(C,0)$.  Moreover, if $g'_1,\ldots, g'_s \in \CC[x,y]$ are any
  functions and $(\mathcal{C}',0) \to (\CC^s,0)$ the analogously formed
  deformation of $(C,0)$, then the tangent map $\CC^s \to \mathcal{T}^1$
  is just induced by the quotient 
  $\CC[x,y] \to \CC[x,y]/(f,\partial_x f, \partial_y f)$.  As soon as
  $\CC^s \twoheadrightarrow \mathcal{T}^1$, 
  the family $\mathcal{C}' \to (\CC^s,0)$
  is itself versal.

  \begin{proposition} \label{prop:smoothoverversal}
    Let $(C,0)$ be the analytic germ of a plane curve singularity
    and let $(\mathcal{C},0) \to (\VV,0)$ be an analytically versal
    deformation of $(C,0)$.  For sufficiently small representatives
    $\mathcal{C} \to \VV$, the relative Hilbert scheme 
    $\mathcal{C}_\VV^{[k]}$ is smooth.
  \end{proposition}
  \begin{proof}
    The relative compactified Jacobian over such a family 
    is known to be smooth 
    \cite[Corollary B.2]{FGvS}, 
    so the result follows from Corollary \ref{cor:smoothcompare}. 
    
    The following direct argument was suggested to us by Rahul 
    Pandharipande. 
    Choose
    $\VV \subset \CC[x,y]$ containing an equation $f$ determining $(C,0)$, 
    such that $(\mathcal{C}_\VV,0) \to (\VV,f)$ determines a versal family 
    for this singularity, and 
    such that $T_f \VV$ contains all polynomials of degree $ \le k$. 
    Then $T_f \VV$ projects surjectively onto $\CC[x,y]/I$ for any
    $I$ of colength $k$, hence by Sequence (\ref{eq:defseq}) 
    the space $\mathcal{C}_{\VV}^{[k]}$ is smooth.
    Now let $\overline{\mathcal{C}} \to \overline{\VV}$ be the 
    miniversal deformation.  By versality there 
    are compatible isomorphisms $\VV \cong \overline{\VV} 
    \times (\CC^t,0)$ and 
    $\mathcal{C} \cong \overline{\mathcal{C}} \times (\CC^t,0)$ 
    \cite[p. 237]{GLS}, 
    and hence also     $\mathcal{C}_{\VV}^{[k]} \cong
    \overline{\mathcal{C}}_{\overline{\VV}}^{[k]} \times (\CC^t,0)$.  
    Thus smoothness of the relative Hilbert schemes over any versal deformation
    is equivalent to smoothness of relative Hilbert schemes over the miniversal
    deformation.    
  \end{proof}

  For fixed $I$ of colength $k$, a generic choice of
  $k$-dimensional $V$ ensures surjectivity of the final map in 
  Sequence (\ref{eq:defseq}).  We must now show that some fixed
  $V$ works for all $I$ containing the equation of the curve.

  \begin{lemma}\label{lem:transversetoideals}
    Let $\oO$ be the complete local ring at a point on a reduced curve, and
    let $\bar{\oO}$ be a finite length quotient of $\oO$.  Let
    $W \subset \bar{\oO}$ be a generic $k$ dimensional vector subspace. 
    Then for $\bar{I}$ the image in $\bar{\oO}$ of any ideal
    of colength $\le k$ in $\tilde{\oO}$, we have $W + \bar{I} = \bar{\oO}$.
  \end{lemma}
  \begin{proof}
    We employ the semigroup of the curve.  
    Fix a normalization $\oO \subset \CC[[t]]^{\oplus r}$.  Define
    \[\ord: \CC[[t]]^{\oplus r} \setminus \{ \mbox{ zero divisors } \} 
    \to \NN^{\oplus r}\] 
    which takes an r-tuple of power series to the r-tuple of degrees of 
    leading elements.  Removing the zero divisors ensures this is well 
    defined. 
    
    All colength $k$ ideals will contain the $k$th power of the maximal
    ideal $M$. 
    Let $\bar{\oO} = \oO/M^k$, and $\Sigma = \ord(\oO) \setminus \ord(M^k)$.  
    If $\ord(f) = \ord(g)$, then some linear 
    combination of $f$ and $g$ has higher order.  
    Therefore we may choose
    a vector space basis of $\bar{\oO}$ of the form
    $\{f_s, s\in \Sigma | \ord(f_s) = s\}$.  For any $\Delta \subset \Sigma$, 
    we define the projection $\pi_\Delta:\bar{\oO} \to 
    \mathrm{Span}\,\{f_s\}_{s \in \Delta}$ by 
    \[
    \pi_\Delta: \sum_{s \in \Sigma}  c_s f_s  \mapsto  \sum_{s \in \Delta} c_s f_s
    \]
    
    Fix an ideal $\tilde{I} \subset \oO$ and write $\overline{I} = I/M^k$. 
    Let $\iota = \ord(\tilde{I})$.  Then $\pi_{\iota \cap \Sigma}|_{\overline{I}}$ is 
    a vector space isomorphism.  
    Thus for a sub vector space $W \subset \bar{\oO}$, we have 
    $W + \bar{I} = \bar{\oO}$  if and only if
    $\pi_{\Sigma \setminus \iota}|_W$ is surjective.
    This determines a Zariski open locus in the Grassmannian of 
    $(\dim W)$ dimensional subspaces of $\bar{\oO}$, which is
    nonempty  if
    $\dim W \ge \# \Sigma \setminus \iota = \dim \bar{\oO}/\bar{I}
    = \dim \oO / (\tilde{I} + M^k)$.  It suffices for
    $\dim W \ge \dim \oO/\tilde{I}$. 
    
    Thus requiring that $\pi_{\{s_1,\ldots,s_k\}}|_W$ is surjective for all
    $\{ s_1, \ldots, s_k \} \subset \Sigma$ 
    ensures that $W$ is transverse to all
    ideals of colength bounded by $k$.  
    The intersection of these finitely many
    nonempty Zariski open sets remains a nonempty Zariski open set.
  \end{proof}


  \begin{theorem} \label{thm:smooth}
    Let $(C,0)$ be the analytic germ of a plane curve singularity
    and let $(\mathcal{C},0) \to (\VV,0)$ be an analytically versal
    deformation of $(C,0)$. Then, for sufficiently small representatives
    $\mathcal{C} \to \VV$, and generic discs $0 \in \DD^k \subset \VV$,
    the space $\mathcal{C}_{\DD^k}^{[h]}$ is smooth for $h \le k$.
  \end{theorem}
  \begin{proof}
    As in Proposition \ref{prop:smoothoverversal}, 
    it suffices to show
    this for any versal deformation $\VV$.  Let $(C,0)$ be given
    by the germ at the origin of the zero locus of 
    $f \in \CC[x,y]$, and choose $g_1,\ldots,g_\tau$ whose images
    in $\CC[[x,y]]/(f,\partial_x f, \partial_y f)$ form a basis; as
    discussed above the miniversal deformation 
    $\mathcal{C} \to \VV = \CC^\tau$ has as fibres the curves
    $f+\sum t_i g_i = 0$.  Let $0 \in \DD^k \subset \VV$ be a generic,
    $k$-dimensional disc.  Lemma \ref{lem:transversetoideals} ensures
    that the image of its tangent space in 
    $\CC[[x,y]]/(f,\partial_x f, \partial_y f)$ is complementary to any
    ideal of colength $h \le k$.  Thus the final map of Sequence 
    (\ref{eq:defseq}) is surjective, and  
    $\mathcal{C}_{\DD^k}^{[h]}$ is smooth at points over $0 \in \DD^k$ 
    which correspond to subschemes supported at the singularity.  
    Finally let $z \subset \mathcal{C}_0$ be any subscheme of length $h$;
    let $z'$ be its component supported at the singularity, say of length
    $h'$.  Then
    an analytic neighborhood of $z$ in $\mathcal{C}_{\DD^k}^{[h]}$ differs
    from an analytic neighborhood of $z'$ in $\mathcal{C}_{\DD^k}^{[h']}$
    by a smooth factor.
  \end{proof}

  \begin{corollary} \label{cor:smooth}
    Let $\mathcal{C} \to \Lambda$ be a family of integral locally planar
    curves, locally versal at $\lambda \in \Lambda$.  Then
    for any generic, sufficiently small $\lambda \in \DD^k \subset \Lambda$, 
    the relative Hilbert scheme $\mathcal{C}^{[h]}_{\DD^k}$ is smooth for
    $h \le k$. 
  \end{corollary}
  \begin{proof}
    This situation is analytically locally smooth over that in the 
    theorem; a compactness argument yields smoothness uniformly
    over an open neighborhood on the base.
  \end{proof}

\section{ADE singularities} \label{sec:examples}

A singularity is said to be simple if it has no nontrivial equisingular
deformations.  Simple singularities of hypersurfaces famously fall into an
ADE classification \cite{AGV}: 

\begin{itemize}
  \item $A_n: y^2 + x^{n+1}$
  \item $D_n: xy^2 + x^{n-1}$
  \item $E_6: y^3 + x^4$ 
  \item $E_7: y^3 + yx^3$ 
  \item $E_8: y^3 + x^5$         
\end{itemize}

We calculate the Euler numbers of some related Hilbert schemes. 
Consider the non-reduced germs at the origin 
$A_\infty: y^2 = 0$, $D_\infty: xy^2 = 0$,
and $E_\infty: y^3 = 0$.  In each case the curve is preserved by the 
full $\CC^* \times \CC^*$ action on $\CC^2$.  The action lifts
to the Hilbert schemes, and the fixed points are monomial ideals in
$\CC[[x,y]]$ containing the equation.  Counting fixed points
gives the following formulas:

\begin{eqnarray*}
  \sum q^n \chi(A_{\infty}^{[n]}) & = & \frac{1}{(1-q)(1-q^2)}  \\
  \sum q^n \chi(D_{\infty}^{[n]}) & = & \frac{1-q+q^3}{(1-q)^2(1-q^2)} \\
  \sum q^n \chi(E_{\infty}^{[n]}) & = & \frac{1}{(1-q)(1-q^2)(1-q^3)} 
\end{eqnarray*}

We now observe that the equation for any simple singularity is
equal to its ``$\infty$'' version modulo $(x,y)^{\delta}$, where 
$\delta$ is the delta invariant of the singularity.  In particular,
the first $\delta$ punctual Hilbert schemes are equal as subvarieties
of the punctual Hilbert scheme of $\CC^2$ at the origin. 
By Corollary \ref{cor:singsupport}, their Euler characteristics suffice to 
determine the whole series.  Explicitly, we have:

\begin{itemize}
  \item $A_{2\delta-1}$: $n_h = {\delta+h  \choose \delta-h}$
  \item $A_{2\delta}$: $n_h = {\delta+h + 1 \choose \delta-h}$
  \item $D_{2\delta-2}$: $n_h = {\delta+h - 3 \choose \delta-h} +
    2{\delta+h-3 \choose \delta-h-1} + {\delta+h- 2 \choose \delta-h-2}$
  \item $D_{2\delta-1}$: $n_h = {\delta+h - 2 \choose \delta-h} +
    2{\delta+h-2 \choose \delta-h-1} + {\delta+h- 1 \choose \delta-h-2}$
  \item $E_6$: $(n_0, \ldots, n_3) = (5, 10, 6, 1)$
  \item $E_7$: $(n_0,\ldots,n_4) = (2, 11, 15, 7, 1)$
  \item $E_8$: $(n_0,\ldots,n_4) = (7, 21, 21, 8, 1)$
\end{itemize}

Theorem A' asserts these numbers are the multiplicities of the Severi
strata.  We present a heuristic argument computing these multiplicities
directly.  To an ADE singularity $c$ is associated the Dynkin diagram with
of the same name.  Its points form a natural basis of vanishing
cycles \cite{AGV}.  
Generic points in $\VV_h$ correspond to curves with $\delta-h$ 
nodes.  As these singular curves are deformations of $c$,
only vanishing cycles of $c$ can collapse at the nodes.  
Moreover, simultaneously contracting intersecting cycles yields singularities
worse than nodes.   Thus the multiplicity of $\VV_h$ is 
the number of different ways to pick $\delta-h$ disjoint vanishing cycles, 
or equivalently
$\delta - h$ vertices of the Dynkin diagram so that no two are connected.
The resulting numbers are precisely the ones given. 

We expect this argument can be made rigorous by using 
either the description of the discriminant of the versal deformation of
an ADE singularity in terms of the associated Weyl group and root lattice
\cite{AGV}, 
or using Grothendieck's classification of the degenerations of 
ADE singularities in terms of the Dynkin diagrams
\cite{Dem}.  Such results seem to be completely 
out of reach for general singularities.
On the other hand, there are so-called D-diagrams attached to all curve 
singularities \cite{AGV}; we would be extremely interested to learn
of a  procedure to compute the Severi degrees from
the D-diagrams.



\section{BPS numbers} \label{sec:bps}

Our original motivation for considering the series on the left hand side of
Theorem A comes from certain curve counting theories on three-folds.  We now
briefly sketch this connection; further details may be found in the papers of
Pandharipande and Thomas \cite{PT1, PT3}.  

For $Y$ a Calabi-Yau three-fold, a parameter count 
suggests that only finitely many genus $g$ curves
will represent any given homology class $\beta \in \mathrm{H}_2(Y)$. 
In fact, the curves may come in positive-dimensional families; 
nonetheless, the {\em Gromov-Witten invariants} are defined
to be the degree of the virtual fundamental class of the
Kontsevich moduli space of stable maps \cite{B, BF, LT}.
These invariants suffer from two major failings: 
first, they are fractional due to the stack structure on the moduli space; 
second, maps from genus $g$ curves will give rise to undesirable maps
from genus $h > g$ curves due to ramified covers 
and collapsing of components.  Conjecturally, both
problems may be simultaneously eliminated by the Faber-Pandharipande \cite{FP}
multiple cover formula, which repackages the Gromov-Witten numbers into
conjecturally integral invariants $n_{h,\beta}^{GW}(Y)$. 
\[\sum_{\beta \ne 0} \sum_{h = 0}^\infty 
\deg\, [\overline{\mathcal{M}}_g(Y,\beta)]^{\mathrm{vir}}
 u^{2h-2} v^\beta = 
\sum_{\beta \ne 0} \sum_{h = 0}^\infty n_{h,\beta}^{GW}(Y)\, u^{2h-2} \sum_{k \ge 1} 
\frac{v^{k\beta}}{k} \left(\frac{\sin(ku/2)}{u/2}\right)^{2h-2}
\]

Gopakumar and Vafa  \cite{GV} 
explained the physical meaning of the $n_{h,\beta}^{GW}(Y)$.
They consider M2-branes in the M-theory in the space 
$\RR^{4,1} \times Y$, i.e., real 3-dimensional manifolds whose projection to
$Y$ is a complex curve in the class $\beta$ and whose projection to $\RR^{4,1}$
is the world-line of a particle.  Integrating out the Calabi-Yau degrees of
freedom suggests that at low energy, the state space of the particle is the 
cohomology of the relative compactified Jacobian of the family of
embedded curves
in class $\beta$.  
\noindent We write this as $\mathrm{H}^*(\mathcal{M}_{\mathrm{GV}})$. 
The theory transforms under $SO(4,\RR) = 
\mathrm{SU}(2)_L \times \mathrm{SU}(2)_R$; in
particular, this group should act $\mathrm{H}^*(\mathcal{M}_{\mathrm{GV}})$. 
The $\mathrm{SU}(2)_R$ induces a weight grading. Forgetting 
the action and collapsing
the grading -- the odd graded pieces become negative virtual
$\mathrm{SU}(2)_L$ representations -- yields
$\mathrm{H}^*(\mathcal{M}_{\mathrm{GV}})|_L \in \mathrm{Rep}(\mathrm{SU}(2)_L)$.  
Enumerative invariants may be extracted via the prescription
\[\mathrm{H}^*(\mathcal{M}_{\mathrm{GV}})|_L = \sum_h n_{h,\beta}^{GV}(Y) \,
(\CC \oplus V_{\mathrm{std}} \oplus \CC)^{\otimes h}\]
In a certain limit the M2-branes become strings, and the 
$n_{h,\beta}^{GV}(Y)$ are related to the Gromov-Witten invariants by precisely the 
multiple cover formula.  That is, 
$n_{h,\beta}^{GV}(Y) = n_{h,\beta}^{GW}(Y)$. 
The $n_{h,\beta}^{GV}(Y)$ may be calculated
by computing the kernels of powers of the $SU(2)_L$ raising operator, 
which, at least in
simple cases, is cup product with the class of the relative theta divisor.  In 
\cite{KKV}, it is shown how the Abel-Jacobi map expresses
these traces in terms of the 
Euler numbers of relative Hilbert schemes of points.  According to
Kawai \cite{Ka}, the Hilbert schemes should be interpreted as
moduli of D2-D0 branes.

The moduli of D2-D0 branes is made mathematically precise in
the work of Pandharipande and Thomas \cite{PT1, PT3}.  They define
\[P_n(Y,\beta) = \{[\phi:\oO_Y \to F]\,| \mbox{$F$ pure, }
\chi(F) = n ,\, [\mathrm{supp}(F)] = \beta,\,
\dim_\CC F/\phi(\oO_Y) < \infty \} \]
This space carries a virtual class $[P_n(Y,\beta)]^{\mathrm{vir}}$
of dimension zero.  Integers $n_{h,\beta}^{PT}$ are defined by 
\[\log \left(1 + \sum_{\beta \ne 0} \sum_n (-q)^n 
v^\beta \deg [P_n(Y,\beta)]^{\mathrm{vir}} \right)  = 
\sum_{h > - \infty} \sum_{\beta \ne 0} n_{h,\beta}^{PT}(Y)\, \sum_{k\ge 1}  
\frac{v^{k\beta}}{k} \left(q^{-k/2} - q^{k/2}\right)^{2h-2}\]
It is conjectured \cite{MNOP, PT1} that $n_{h,\beta}^{PT}(Y) = n_{h,\beta}^{GW}(Y)$.

We consider only irreducible $\beta$; for these the expression simplifies to 
\[\sum_n (-q)^n  \deg [P_n(Y,\beta)]^{\mathrm{vir}}  = 
\sum_{h > - \infty} n_{h,\beta}^{PT}(Y)\, \left(q^{-1/2} - q^{1/2}\right)^{2h-2}\]
In \cite{PT1,PT3}, it is observed that the $P_n$ carry symmetric
perfect obstruction theories; Behrend \cite{B2} has shown the resulting
virtual degrees can be computed as 
$\deg [P_n(Y,\beta)]^{\mathrm{vir}} = \chi(P_n(Y,\beta), \nu^b)$. 
Here $\nu^b$ is a constructible function depending only on the scheme 
structure in an analytic
local neighborhood, and not on the obstruction theory.  
This makes it possible to discuss the contribution of a single curve.
That is, if $\mathcal{C} \to \Lambda$ is the family
of curves in class $\beta$, then for $\lambda \in \Lambda$ we define 
$n_{h,\beta}^{PT}(\mathcal{C}_\lambda)$ by
\[\sum_n (-q)^n  \,\chi\! \left(P_n(\mathcal{C}_\lambda), 
\nu^b|_{P_n(\mathcal{C}_\lambda)}\right)  = 
\sum_{h > - \infty} n_{h,\beta}^{PT}(\mathcal{C}_\lambda)  
\left(q^{-1/2} - q^{1/2}\right)^{2h-2}\]
Here, the space $P_n(\mathcal{C}_\lambda) \subset P_n(Y,\beta)$ is 
the locus where the sheaf $F$ is (scheme-theoretically) supported
on the curve $\mathcal{C}_\lambda$.  
The function $\nu^b|_{P_n(\mathcal{C}_\lambda)}$ is restricted from $P_n(Y,\beta)$
and {\em is not} intrinsic to $P_n(\mathcal{C}_\lambda)$.  
If $n_h^{PT}:\Lambda \to \ZZ$ is the function 
$\lambda \mapsto n_h^{PT}(\mathcal{C}_\lambda)$, then 
$n_{h,\beta}^{PT}(Y) = \chi(\Lambda,n_h)$.  

Assume $\mathcal{C}_\lambda$ is integral and locally planar. 
Then \cite[Appendix B]{PT3}, 
since $\mathcal{C}_\lambda$ is Gorenstein, we can identify
$P_{n+1-g}(\mathcal{C}_\lambda) = \mathcal{C}_\lambda^{[n]}$.  
It follows from Corollary
\ref{cor:smoothcompare} 
that if the total space of 
the relative compactified Jacobian of
the family $\mathcal{C} \to \Lambda$ is smooth at points over $\lambda$, then 
$\nu^b|_{P_n(\mathcal{C}_\lambda)} = (-1)^{n-1+g+\Lambda}$.  
This certainly holds at points where $\Lambda$ is smooth and 
$\mathcal{C} \to \Lambda$ is locally versal; in fact \cite{FGvS}, 
it suffices for its image in the product of the versal deformations
of the singularities to be transverse to the tangent cone of the
equigeneric stratum.
In this case,  
$(-1)^{\dim \Lambda} n_{h}^{PT}(\mathcal{C}_\lambda) = n_h(\mathcal{C}_\lambda)$; 
the left hand side being the invariants discussed in this article.  

The $n_h^{PT}(\mathcal{C}_\lambda)$ should count the ``number
of curves of geometric genus $h$ occuring at $\lambda$''.  
In the situation we have been discussing, 
Theorem A gives a sense in which this is true.

\section{The HOMFLY polynomial of the link} \label{sec:knots}

A knot is a smooth embedding $S^1 \to S^3$,
considered up to isotopy; more generally, a link is an smooth embedding
of possibly several circles.  
Singularities naturally give rise to links.  
If $p \in C \subset S$ is a point on a curve on a surface, 
and $B_\epsilon(p)$ is a small ball containing $p$, then we write
$\mathrm{Link}(C,p)$ for  
$C \cap \partial B_\epsilon(p) \subset \partial
B_\epsilon(p)$.  
Data about the singularity is reflected in the topology of the link; 
for instance, 
the link is trivial iff the $p$ is a smooth point, and the number of components 
of the link is equal to the number of analytic local branches at $p$. 
In fact \cite{Z}, the link determines
the equisingularity class of the germ of $C$ at $p$. 
For discussions of the interplay between singularities and knots,
see \cite{M, AGV, Wa}.

A central project of knot theory is the classification of knots and links by
means of invariants.  Given the close relationship between a singularity
and its link, one may ask 
what various topological invariants of the link capture about
the geometry the singularity, and, conversely, what algebro-geometric invariants
say about the topology.  For example, Campillo, Delgado, and Gusein-Zade have
proven that the multivariate Alexander polynomial of the link is a certain graded
Euler number of the ring of functions at the singular point \cite{CDG}. 
It is known that the link type in
turn can be recovered from the multivariate Alexander polynomial \cite{Y}. 

There is a generalization of the (usual univariate) 
Alexander polynomial, 
variously called the skein, Jones-Conway, HOMFLY, or HOMFLY-PT
polynomial \cite{HOMFLY}.  We denote it by $\mathbf{P}$. 
It associates an element of $\ZZ[a^{\pm 1}, z^{\pm 1}]$ to any oriented link,
and is characterized by its
behavior when strands of the link pass through one another:  
\begin{eqnarray*}
  \label{eq:skein1}
  a^{-1} \, 
  \mathbf{P}(\overcrossing)  -
  a \, \mathbf{P}(\undercrossing)  & = & z
  \, \mathbf{P}(\smoothing) \\
  \label{eq:skein2} 
  a^{-1}-a & = & 
  z\,  \mathbf{P}(\bigcirc) 
\end{eqnarray*}
We write $P_\infty \in \ZZ[z^{\pm 1}]$ for the coefficient of the lowest power
of $a$. 

Suppose $C$ is rational with a unique singularity at $p$. 
Oblomkov and the present author \cite{OS} 
have conjectured a relation between
the Hilbert schemes of points on $C$ and 
the HOMFLY polynomial of the link of $C$ at $p$.  Here we state only
its specialization to $P_\infty$:

\begin{conjecture} \label{conj:OS} (Oblomkov--Shende \cite{OS}.)
Let $p \in C$ be a point on a locally planar curve; let $c$ denote
the analytic germ at this point.  Let $c$ have 
have $b$ branches and contribute $\delta$ to the arithmetic genus.
Then, 
\[P_\infty(\mathrm{Link}(C,p)) = \sum_{h=0}^\delta n_h(c) z^{2h-b}\]
\end{conjecture}


Theorem A' gives an enumerative interpretation of the coefficients on the
right hand side.  One may ask whether any such meaning exists for 
the analogous coefficients on the left hand side.  We will
find one in the work of Jaeger \cite{Ja}.

Recall that the braid group is $\pi_1(\mathrm{Conf}_n(\CC),\star)$, where
\[\mathrm{Conf}_n(\CC) = \{\,n \,\mbox{unlabelled, distinct points in } \CC\, \}\]
At the basepoint $\star\in \CC^{(n)}$, we label the $n$ points as
$p_1, \ldots p_n$.  The braid group is generated
by $\tau_1,\ldots,\tau_{n-1}$, where $\tau_i$ is the 
counter-clockwise half-twist interchanging $p_i$ and $p_{i+1}$ while
leaving all other points fixed.
Their inverses are the analogous clockwise half-twists.  The relations are 
generated by $\tau_i \tau_j = \tau_j \tau_i$ if $|i-j| \ne 1$ and 
$\tau_i \tau_{i+1} \tau_i = \tau_{i+1} \tau_i \tau_{i+1}$.  
A braid may be ``closed'' to form an oriented link; this is done by 
associating to a loop 
$S^1 \to \CC^{(n)}$ its evaluation graph in the solid torus
$S^1 \times \CC$, and then embedding the solid torus in the usual
way into $S^3$.  The orientation lifts from the orientation of $S^1$. 
That any link may be obtained in this manner is
a classical theorem of Alexander \cite{Al}.

We now describe Jaeger's formula.  Fix some sequence
$\tau_{i_1}^{\pm 1} \ldots \tau_{i_N}^{\pm 1}$.  We denote both the sequence
and its product braid by $\beta$.   Consider now the set of all sequences
formed from $\beta$ by replacing some of the 
$\tau_i$ with symbols ${\not} \tau_i$ and likewise 
some $\tau_i^{-1}$ with ${\not} \tau_i^{-1}$. Jaeger calls these
``circuit partitions''.  Such a sequence determines an element of the
braid group by viewing all ${\not} \tau_i^{\pm 1}$ as identity elements. 

Consider tracing through the braid closure of the 
new sequence in the following manner.  Start at the point 
$(1,p_1) \in S^1 \times \CC$, and move according to the orientation lifted
from the circle.   While travelling, keep track of the {\em strand number},
which begins at $1$ and when passing $\tau^{\pm 1}_i$ is changed by
the transposition $i \leftrightarrow i+1$.  
Continue until returning to $(1, p_1)$.  If there are multiple link components,
now jump to the
first point $(1,p_k)$ which has not yet been encountered, set the strand
number to $k$, and continue. 
Along this path, each of the half-twists or removed half-twists is encountered
twice.  The sequence is {\em admissible} if 
the first encounter of a given ${\not} \tau$ (resp. ${\not} \tau^{-1}$) 
has lower (resp. higher) strand number
than the second.

We write $A(\beta)$ for the set of admissible sequences.  
Denote by $w(\beta)$ the {\em writhe}, i.e., the
number of $\tau$ minus the number of $\tau^{-1}$; it does not depend
on the presentation of the braid.  
Given 
$\pi \in A(\beta)$, let 
$b(\pi)$ denote the number of components of the braid closure of
the braid associated to $\pi$.  
Jaeger proves \cite{Ja}:\footnote{Our expression is slightly different due to 
a different convention for the HOMFLY polynomial.}

\[\mathbf{P}(\overline{\beta}) = a^{w(\beta)} \sum_{\pi \in A(\beta)} 
(-1)^{\# \not{\tau^{-1}}} z^{\# \not{\tau}} a^{n-b(\pi)} \mathbf{P}(\bigcirc)^{b(\pi)} \]

\begin{example}  Consider the braid on $n=2$ strands given
by $\tau_1^3$; the corresponding knot is the right handed trefoil which
is the link of an ordinary cusp.  Then the admissible sequences are 
${\not} \tau_1 {\not} \tau_1 {\not} \tau_1$, 
${\not} \tau_1 {\not} \tau_1 \tau_1$, 
$\tau_1 \tau_1 {\not} \tau_1$, 
${\not} \tau_1 \tau_1 \tau_1$, 
$\tau_1 \tau_1 \tau_1$.  The resulting formula for the
HOMFLY polynomial is 
\[
\mathbf{P}(\mathrm{trefoil}) =
a^{3} (z^3 \mathbf{P}(\bigcirc)^2 + z^2 a \mathbf{P}(\bigcirc) + 
2z \mathbf{P}(\bigcirc)^2  + a \mathbf{P}(\bigcirc)) =  
a^2 (2- a^2 + z^2) \mathbf{P}(\bigcirc)
\]
Thus $P_{\infty}(\mathrm{trefoil}) = 2z^{-1} + z$, matching 
the $n_0(\mathrm{cusp}) = 2$ and $n_1(\mathrm{cusp}) = 1$
observed in Corollary \ref{cor:nodalcubic}.
\end{example}

Henceforth we discuss only {\em positive} braids, i.e., those
which are products of counter-clockwise half twists.  The
description of links of singularities as iterated torus knots
(see e.g. \cite{Wa}) yields positive braid presentations.  
For such braids the writhe $w$ is just the number of twists appearing, 
and it can be seen from Jaeger's formula that the number $w - n$ is 
an invariant of the closed braid. 
Since $\tau^{\pm 1}$ changes the number of link components by $1$, 
any circuit partition $\pi$ with $b(\pi) = n$ must have an even number
of $\tau^{\pm 1}$. Denote the set of admissible circuit partitions
with $2r$ half twists by
$A_{n,r}(\beta)$.   Counting link components, 
we see that 
$n - b(\beta) \le w - 2r$.  Restricting Jaeger's formula to 
the lowest degree term in $a$,
\[P_{\infty}(\overline{\beta}) \,\, = \!\!\! \sum_{r = 0}^{(w+b- n)/2} 
 \!\!\!\!\# A_{n,r}(\beta)\, z^{w - n - 2r}\]

\begin{remark}  
$P_\infty(L)$ has non-negative coefficients if $L$ 
admits a positive braid presentation.  Thus 
Conjecture \ref{conj:OS} predicts that 
$n_h(C) \ge 0$.  The identification of the $n_h(C)$ as multiplicities
establishes this positivity. 
\end{remark}

\begin{example}  Consider the braid on $n=3$ strands given
by $(\tau_1 \tau_2)^4$; the corresponding knot is the right handed (3,4) torus
knot, which is the link of the E6 singularity.  Let us abbreviate
$\tau_1 = \tau$ and $\tau_2 = \sigma$.  The admissible sequences $\pi$
with $b(\pi) = 3$ are: 
\[
\begin{array}{c|l}
  T_{3,4} & \mbox{admissible sequences} \\
  \hline
  A_{3,0} & \begin{array}{l} ({\not} \tau {\not} \sigma)^4 \end{array} \\
  \hline
  A_{3,1} & 
  \begin{array}{l}
    (\tau {\not} \sigma)^2 ({\not} \tau {\not} \sigma)^2,
    ({\not} \tau {\not} \sigma)(\tau {\not} \sigma)^2
    ({\not} \tau {\not} \sigma),
    ({\not} \tau {\not} \sigma)^2(\tau {\not} \sigma)^2 \\
    ({\not} \tau  \sigma)^2 ({\not} \tau {\not} \sigma)^2,
    ({\not} \tau {\not} \sigma)({\not} \tau  \sigma)^2
    ({\not} \tau {\not} \sigma),
    ({\not} \tau {\not} \sigma)^2({\not} \tau  \sigma)^2 
  \end{array} 
  \\
  \hline
  A_{3,2} & 
  \begin{array}{l}
    (\tau {\not} \sigma)^4 ,     ({\not}  \tau \sigma)^4  \\
    \tau {\not} \sigma \tau {\not} \sigma {\not} \tau
    \sigma {\not} \tau \sigma, 
    \tau {\not} \sigma \tau \sigma {\not} \tau \sigma {\not} \tau
    {\not} \sigma, 
    {\not} \tau {\not} \sigma \tau {\not} \sigma \tau 
    \sigma {\not} \tau \sigma,
    {\not} \tau \sigma {\not} \tau \sigma \tau {\not} \sigma \tau {\not}\sigma
    \\
    \tau \sigma {\not} \tau \sigma \tau {\not} \sigma {\not} \tau {\not} \sigma,
    {\not} \tau {\not} \sigma \tau \sigma {\not} \tau \sigma \tau {\not} \sigma,
    {\not} \tau \sigma \tau {\not} \sigma \tau \sigma {\not} \tau {\not} \sigma,
    {\not} \tau {\not} \sigma {\not} \tau \sigma \tau {\not} \sigma \tau \sigma
  \end{array} 
  \\
  \hline
  A_{3,3} & 
    \begin{array}{l}
      (\tau {\not} \sigma \tau \sigma)^2 ,
      (\tau \sigma  {\not} \tau \sigma)^2 ,
      (\tau \sigma)^3 {\not} \tau {\not} \sigma,\,
      {\not} \tau {\not} \sigma (\tau \sigma)^3, 
      {\not} \tau  (\sigma \tau)^3 {\not} \sigma
    \end{array}
    \\ 
  \end{array}
\]
Jaeger's formula gives 
$P_\infty(T_{3,4}) = z^{-1} + 6 z + 10 z^3 + 5 z^5$, matching
the values for $n_h(E_6)$ in Section \ref{sec:examples}. 
\end{example}

\begin{lemma}
  Consider a singularity with $b$ analytic local branches,
  delta invariant $\delta$, and Milnor number $\mu$.  Let
  $\beta$ be a positive braid presentation of the 
  link of the singularity, with $n$ strands and $w$ crossings. 
  Then $\mu = w - n + 1$, or equivalently, 
  $2 \delta = w - n + b$. 
\end{lemma}
\begin{proof}
  Specialize Jaeger's formula to the Alexander polynomial, and use known
  properties relating its degree with the Milnor number of the singularity
  \cite{M}.
\end{proof}

Fix a singularity, $c$, and a positive braid presentation $\beta$ of 
its link.  Let $\VV_{\delta-r}^+ \subset \VV(c)$ denote as usual the locus
of deformations of $c$ with $r$ nodes and no other singularities.  
Let $\DD^r$ be a generic disc in $\VV(c)$.
In light of Theorem A, 
Conjecture \ref{conj:OS} is equivalent to the
assertion that $\# A_{n,r} = \VV_{\delta - r} \cap \DD^r$.  
For $r = 0, 1$, this is straightforward.  There is evidently 
a unique element of $A_{n,0}$, and as $\VV_\delta = \VV(c)$, a generic space of 
complementary dimension is a single point.  An element
of $A_{n,1}$ must have two remaining $\tau_i$,
for some fixed $i$.  The admissibility condition ensures that the second
$\tau_i$ must be the first one occuring in the original $\beta$ after the
first $\tau_i$.  Thus $\# A_{n,1} = w - (n-1)$.  This is equal to
the Milnor number
of the singularity, which is in turn equal to the multiplicity of the 
discriminant locus. 

We now speculate about how a bijection may be established 
between $A_{n,r}$ and $ \VV_{\delta - r} \cap \DD^r$.  That is,
how to match deformations with $r$ nodes 
to circuit partitions with $2r$ remaining 
half-twists.  View $c$ as the germ at the origin of some curve $C$ in
$\CC^2$.  Choose
a projection $\CC^2 \to \CC$; it induces a finite map $C \to \CC$.  
In a small punctured disc $\DD^* \subset \CC$ the map is unramified;
say it has degree $n$.  Thus the boundary of the disc gives a $S^1$ family of
$n$ points moving in $\CC$; the closure of the corresponding braid 
$\beta$ is the 
link of the singularity.  Now deform $c$ very slightly to a smooth curve
$c_0$ whose projection to $\DD$ has only simple 
ramification, say at points in $\mathcal{R} \subset \DD$.  
Comparing Euler numbers reveals 
$\#\mathcal{R} = 2 \delta - b + n = w$. 

Above each point in $\DD \setminus \rR$ is a collection of $n$ points in
the fibre, which is $\CC$.  Fix a point $d$ on the boundary of the disk
and let $\star \in \mathrm{Conf}_n(\CC)$ be the points in $c_0$ lying
over it.  We can thus form the braid monodromy \cite{Mo}: 
\[BM:\pi_1(\DD \setminus \rR, d) \to \pi_1(\mathrm{Conf}_n(\CC), \star)\]  
The image of 
a loop containing no ramification points is the trivial braid; 
the image of the loop $\partial \DD$ containing all the ramification points is 
a braid whose
closure is the link of the singularity.  The image of 
a loop containing exactly one ramification point
is a braid which interchanges two points in the fibre 
by a positive half-twist.  
The description of $\beta$ as an iterated torus link gives a 
positive braid presentation on $n$ strands; this presentation
must have exactly $w = \mu + n - 1$ half-twists.  As 
$w = \# \mathcal{R}$, we find it plausible that there exists
a decomposition $\partial \DD = \ell_1 \cdots \ell_w$ into loops 
$\ell_j$ containing one ramification point each, such that 
$BM(\ell_j) \in \{\tau_1,\ldots,\tau_{n-1}\}$. 
Fix such a decomposition.

Consider an intersection point of a generic hyperplane with 
$\VV_{\delta-h}$.  This corresponds to a curve 
$c_h$ with exactly $h$ nodes. 
By genericity $c_h$ projects to $\DD$ with only simple ramification, 
and with no nodes over the ramification points.
We denote the ramification points by $\mathcal{R}(c_h) \subset \DD$, 
and the images of the nodes by $\mathcal{N}(c_h)$. 
Evidently $\# \mathcal{R} - \# \mathcal{R}(c_h) = 2h$. 
We again have the braid monodromy,
\[BM:\pi_1(\DD \setminus (\rR(c_h) \cup \nN(c_h))) \to 
\pi_1(\mathrm{Conf}_n(\CC),\star)\]
Choose a path in $\VV(c)$ from $c_0$ to $c_h$.  Traversing this path,
some ramification points will remain ramification points and others will
collide to form nodes; thus we have a map 
$\phi: \rR \to \rR(c_h) \cup \nN(c_h)$.  
We define a circuit partition $\pi(c_h)$ by taking the sequence
$BM(\ell_1) \cdots BM(\ell_w)$ and replacing the
$BM(\ell_i)$ for $i \notin \phi^{-1}(\nN(c_h))$ with ${\not} \tau$ terms.
The chosen path induces an inclusion 
$\pi_1(\DD \setminus (\rR(c_h) \cup \nN(c_h))) \hookrightarrow
\pi_1(\DD \setminus \rR)$ compatible with the braid monodromy; by
construction the braid associated to $\pi(c_h)$ comes from a loop
in $\pi_1(\DD \setminus (\rR(c_h) \cup \nN(c_h)))$ which goes around
all the nodes and none of the ramification points.  Thus the braid
closure has $n$ components.

The admissibility of $\pi(c_h)$ presumably depends on the 
path chosen from $c_0$ to $c_h$; we do not know how to choose 
paths in a systematic way.  Even having done this, one
must somehow show every admissible circuit partition occurs exactly once.  
We leave the further study of these matters to future work.

\end{document}